\documentclass[letterpaper, 10 pt, conference]{ieeeconf}
\IEEEoverridecommandlockouts
\usepackage{url}
\usepackage{amsmath, amsfonts}

\def\ao{}
\def\aa{}

\newtheorem{theorem}{\indent {Theorem}}

\newtheorem{lemma}{\indent {Lemma}}
\newcommand{\qed}{\ \ {\bf q.e.d.}}

\def\0{{\bf 0}}
\def\1{{\bf 1}}
\def\R{\mathbb{R}}

\providecommand{\com[2]}{\begin{tt}[#1: #2]\end{tt}}

\begin{document}

\title{A lower bound for distributed averaging algorithms}
\author{Alex Olshevsky, John N. Tsitsiklis\thanks{The authors are with
the Laboratory for Information and Decision Systems, Massachusetts
Institute of Technology, Cambridge, MA, USA. Email:
{\tt $\{$alex\_o,jnt$\}$@mit.edu}. This research was supported by the
NSF under grants ECCS-0701623 and CMMI-0856063.}}

\maketitle
\begin{abstract} \ao{ We derive lower bounds on the convergence
speed of a widely used class of distributed averaging algorithms. In particular, we prove that any distributed averaging algorithm whose state consists of a single real number and whose (possibly nonlinear) update function satisfies a natural smoothness condition has a worst case running time of at least on the order of $n^2$ on a network of $n$ nodes. Our results suggest that increased memory or expansion of the state space is crucial for improving the running times of distributed averaging algorithms. }
\end{abstract}
\section{Introduction}

The goal of this paper is to analyze the fundamental limitations of
a class of distributed averaging algorithms. These algorithms are
message-passing rules
 for a collection of agents (which may be sensors, nodes of a communication network, or UAVs), each beginning with a real number, to estimate the average of these
numbers using only nearest neighbor communications.  Such algorithms
are interesting because a number of sophisticated network
coordination tasks can be reduced to averaging (see
\cite{LR04, XBL05, AH08, BCH08, CCSZ08, GCB08, HU08, SSR08, WSD09}), and also because they can be
designed to be robust to frequent failures of communication links.

 A variety of such algorithms are available (see \cite{TBA86, JLM03, OM04,XB04, M05, MVR06, NOOT07, KBS07, ZM08,  OT09}).  However,  many of these algorithms tend to suffer from a common disadvantage: even when no link failures occur,  their convergence times
 do not scale well in the number of agents. Our aim in this paper is to show that this is, in fact,
 unavoidable for a common class of such algorithms; namely, that any distributed averaging algorithm
 that uses a single scalar state variable at each agent and satisfies a natural ``smoothness'' condition will have
 this property.

 We next proceed to define distributed averaging algorithms and
informally state our result. \subsection{\ao{Background and basic
definitions.}} \noindent {\bf Definition of local averaging
algorithms:} Agents $1,\ldots,n$
 begin with real numbers $x_1(0),\ldots,x_n(0)$ stored in memory. At each round $t=0,1,2,\ldots$, agent $i$ broadcasts $x_i(t)$ to each of
its neighbors in some {\em undirected} graph $G(t)=(\{1,\ldots,n\}, E(t))$, and then sets $x_i(t+1)$ to be some
  function of $x_i(t)$ and of the values $x_{i'}(t), x_{i''}(t), \ldots $ it has just received from its own neighbors:
 \begin{equation} \label{onehopupdate} x_i(t+1) = f_{i, G(t)} ( x_i(t), x_{i'}(t), x_{i''}(t), \ldots ). \end{equation} We require each $f_{i,G(t)}$ to be a differentiable function. Each
 agent uses the incoming messages $x_{i'}(t), x_{i''}(t),\ldots$ as the arguments of $f_{i,G(t)}$ in some arbitrary order; we
 assume that this order does not change, i.e. if $G(t_1)=G(t_2)$, then the message coming from the same neighbor of agent $i$
  is mapped to the same argument of $f_{i,G(t)}$ for $t=t_1$ and $t=t_2$.  It is desired that
 \begin{equation} \label{convergence} \lim_{t \rightarrow \infty} x_i(t) = \frac{1}{n} \sum_{i=1}^n x_i(0),\end{equation} for every $i$,  for every sequence of graphs $G(t)$ having the property that
 \begin{equation} \label{connectivity}  \mbox{ the graph } (\{1,\ldots,n\}, \cup_{s \geq t} E(s)) \mbox{  is connected for every } t, \end{equation} and for every possible way for the agents to map incoming messages to arguments of $f_{i,G(t)}$.

  In words, as the number of rounds $t$ approaches infinity, iteration (\ref{onehopupdate}) must converge to the average of the numbers $x_1(0), \ldots,x_n(0)$. Note that the agents have no control over the communication graph sequence $G(t)$, which is exogenously provided by ``nature." However, as we stated previously, every element of the sequence $G(t)$ must be undirected: this corresponds to bidirectional models of communication between agents. Moreover, the sequence $G(t)$ must satisfy the mild connectivity condition of Eq. (\ref{connectivity}), which says that the network cannot become
 disconnected after a finite period.

 Local averaging algorithms are useful tools for information fusion due to their efficient utilization of
resources (each agent stores only a single number in memory) as well
as their robustness properties (the sequence of graphs $G(t)$ is
time-varying, and it only needs to satisfy the relatively weak
connectivity condition in Eq. (\ref{connectivity}) for the
convergence in Eq. (\ref{convergence}) to hold). As far as the
authors are aware, no other class of schemes for averaging (e.g.,
flooding, fusion along a spanning tree) is known to produce similar
results under the same assumptions.

\bigskip

\noindent {\bf Remark:}  As can be seen from the
subscripts, the update function $f_{i,G(t)}$ is allowed to depend on
the agent and on the graph. Some dependence on the graph is
unavoidable since in different graphs an agent may have a different
number of neighbors, in which case nodes will receive a different
number of messages, so that even the number of arguments of
$f_{i,G(t)}$ will depend on $G(t)$. It is often practically desired
that $f_{i,G(t)}$ depend only weakly on the graph, as the entire
graph may be unknown to agent $i$. For example, we might require
that $f_{i,G(t)}$ be completely determined by the degree of $i$ in
$G(t)$. However, since our focus is on what distributed algorithms
{\em cannot} do, it does not hurt to assume the agents have
unrealistically rich information; thus we will not assume any
restrictions on how $f_{i,G(t)}$ depends on $G(t)$.

\bigskip

\noindent {\bf Remark:}  We require the functions
$f_{i,G(t)}$ to be smooth, for the following reason. First, we
need to exclude unnatural algorithms that encode vector information in the
infinitely many bits of a single real number. Second, although we
make the convenient technical assumption that agents can transmit
and store real numbers, we must be aware that in practice agents
will transmit and store a quantized version of $x_i(t)$. Thus, we are mostly interested
in algorithms that are not disrupted much by quantization. For this reason,
 we must prohibit the agents from using {\em discontinuous}
update functions $f_{i, G(t)}$. For technical reasons, we actually
go a little further, and prohibit the agents from using {\em
non-smooth} update functions $f_{i,G(t)}$.

 \bigskip

%

\subsection{ \ao{Examples.}}
In order to provide some context, let us mention just a few of the
distributed averaging schemes that have been proposed in the
literature:

\begin{enumerate} \item The max-degree method \cite{OM04} involves picking $\epsilon(t)$ with the property $\epsilon(t) \leq 1/(d(t)+1)$,  where $d(t)$ is the largest degree of any agent in $G(t)$, and updating by
\[ x_i(t+1) = x_i(t) + \epsilon(t) \sum_{i \in N_i(t)} \left( x_j(t) - x_i(t) \right) .\] Here we use $N_i(t)$ to denote the set of neighbors of agent $i$ in $G(t)$. In practice, a satisfactory $\epsilon(t)$ may not be known to all of
the agents, because this requires some global information. However,
in some cases a satisfactory choice for $\epsilon(t)$ may be
available, for example when an a priori upper bound on $d(G(t))$ is
known.
\item The Metropolis method \cite{XB04} involves setting $\epsilon_{ij}(t)$ to satisfy $\epsilon_{ij}(t) \leq \min(1/(d_i(t)+1), 1/(d_j(t)+1))$, where $d_i(t), d_j(t)$ are the degrees of agents $i$ and $j$ in $G(t)$, and
    updating by
    \[ x_i(t+1) = x_i(t) + \sum_{j \in N_i(t)} \epsilon_{ij}(t) \left( x_j(t) - x_i(t) \right). \]
    \item The load-balancing algorithm of \cite{NOOT07} involves
    updating by
    \[ x_i(t+1) = x_i(t) + \sum_{i \in N_i(t)} a_{ij}(t) \left( x_j(t) - x_i(t) \right),\] where $a_{ij}(t)$ is determined by the following rule: each agent selects exactly two neighbors, the neighbor with the largest value above its own and with the smallest value below its own. If $i,j$ have both selected each other, then
    $a_{ij}(t)=1/3$; else $a_{ij}(t)=0$. The intuition comes from load-balancing: agents think of $x_i(t)$
     as load to be equalized among their neighbors; they try to offload on their lightest neighbor and take from their heaviest neighbor.

\end{enumerate}

We remark that the above load-balancing algorithm is not a ``local
averaging algorithm'' according to our definition because $x_i(t+1)$
does not depend only on $x_i(t)$ and its neighbors; for example,
agents $i$ and $j$ may not match up because  $j$ has a neighbor $k$
with $x_k(t) > x_j(t)$. By contrast, the max-degree and Metropolis
algorithm are indeed ``local averaging algorithms.''

For each of the above algorithms, it is known that Eq.
(\ref{convergence}) holds provided the connectivity condition in Eq.
(\ref{connectivity}) holds.  A proof of this fact for the
load-balancing algorithm is implicit in \cite{NOOT07}, and for the
others it follows from the results of \cite{LW04, BHOT05}.

\subsection{\ao{Our contribution}}

Our goal is to study the worst-case convergence time over all graph
sequences. \ao{This convergence time may be arbitrarily bad since
one can insert arbitrarily many empty graphs into the sequence
$G(t)$ without violating Eq.\ (\ref{connectivity}). To avoid this
trivial situation, we require that there exist some integer $B$
such that the graphs}
\begin{equation} \label{strongerconnect} (\{1,\ldots,n\}, \cup_{i=kB}^{(k+1)B} E(k))
 \end{equation} are connected for every integer $k$.

Let $x(t)$ be the vector in $\Re^n$ whose $i$th component is $x_i(t)$.
We define the convergence time $T(n,\epsilon)$ of a local averaging algorithm as the \aa{time until
``sample variance'' } \[ V(x(t)) = \sum_{i=1}^n \left( x_i(t)-\frac{1}{n}
\sum_{j=1}^n x_j(0) \right)^2 \] \aa{permanently shrinks by a factor of $\epsilon$, i.e., $V(x(t)) \leq \epsilon V(x(0))$ for all $t \geq T(n, \epsilon)$, for all possible  $n$-node graph sequences satisfying
Eq.\ (\ref{strongerconnect}), and all initial vectors $x(0)$ for
which not all $x_i(0)$ are equal; $T(n,\epsilon)$ is defined to be the smallest number
with this property.} We are interested in how
$T(n,\epsilon)$ scales with $n$ and $\epsilon$.

Currently, the best available upper bound for the convergence time is
obtained with the load-balancing algorithm; in \cite{NOOT07} it was
proven that
\[ T(n,\epsilon) \leq C n^2 B \log \frac{1}{\epsilon},
\] for some absolute constant\footnote{By ``absolute constant'' we
mean that $C$ does not depend on the problem parameters $n, B,
\epsilon$.} $C$.We are primarily interested in whether its possible to improve the
scaling with $n$ to below $n^2$.  Are there  nonlinear update
functions $f_{i,G(t)}$ which speed up the convergence time?

Our main result is that the answer to this question is ``no" within
the class of local averaging algorithms. For such algorithms we
prove a general lower bound of the form \[ T(n,\epsilon) \geq c n^2
B \log \frac{1}{\epsilon}, \] for some absolute constant $c$.
Moreover, this lower bound holds even if we assume that the graph
sequence $G(t)$ is the same for all $t$; in fact, we prove it for
the case where $G(t)$ is a fixed ``line graph.''

\section{Formal statement and proof of main result}

We next state our main theorem. The theorem
begins by  specializing our definition of local averaging
algorithm to the case of a fixed line graph, and states a lower
bound on the convergence time in this setting.

We will use the notation $\1$ to denote the vector in $\R^n$ whose
entries are all ones, and $\0$ to denote the vector whose
entries are all $0$. The average of the initial values
$x_1(0),\ldots,x_n(0)$ will be denoted by $\bar{x}$.

\noindent \begin{theorem} \label{mainthm} Let $f_1$, $f_n$ be two differentiable functions from $\R^2$ to $\R$, and
let $f_2,f_3,\ldots,f_{n-1}$ be differentiable functions from $\R^3$ to $\R$. Consider
the dynamical system \begin{eqnarray} x_1(t+1) & = &
f_1(x_1(t),x_2(t)), \nonumber \\
x_i(t+1) & = & f_i(x_i(t),
x_{i-1}(t),x_{i+1}(t)), ~~ i = 2,\ldots,n-1, \nonumber \\
x_n(t+1) & = & f_n(x_{n-1}(t),x_n(t)). \label{dynsys} \end{eqnarray} Suppose that there exists a function $\tau(n,\epsilon)$ such that
 \[ \frac{ \|x(t)- \bar{x} \1 \|_2}{\|x(0)- \bar{x} \1\|_2}  < \epsilon, \] for all $n$ and $\epsilon$, all $t \geq \tau(n,\epsilon)$, and all initial conditions $x_1(0),\ldots,x_n(0)$ for
which not all $x_i(0)$ are equal. Then,
\begin{equation} \label{mainresult} \tau(n,\epsilon) \geq \frac{n^2}{30} \log \frac{1}{\epsilon},\end{equation} for all $\epsilon > 0$ and $n \geq 3$.
\end{theorem}

\bigskip

\noindent {\bf Remark:} The dynamical system described in the
theorem statement is simply what a local averaging algorithm looks
like on a line graph. The functions $f_1, f_n$ are the update
functions at the left and right endpoints of the line (which have
only a single neighbor),  while the update functions $f_2, f_3,
\ldots, f_{n-1}$ are the ones used by the middle agents (which have
two neighbors). As a corollary, the convergence time of any local averaging algorithm must satisfy  the lower bound
$T(n,\epsilon) \geq (1/30) n^2 \log (1/\epsilon)$.

\bigskip

\noindent {\bf Remark:} \aa{Fix some $n \geq 3$. A corollary of our theorem is that} there are no
``local averaging algorithms'' which compute the average in finite time. \aa{ More precisely,
 there is no local averaging algorithm which, starting from initial conditions $x(0)$ in some
  ball around the origin, always results in $x(t) = {\bar x} \bf 1$ for all times $t$ larger than
 some $T$.}  We will sketch a proof of this after proving Theorem 1. By contrast, the existence of such algorithms in slightly
different models of agent interactions was demonstrated in
\cite{C06} and \cite{SH07}.

\subsection{Proof of Theorem \ref{mainthm}.}

We first briefly sketch the proof strategy. We will begin by pointing out that $\0$ must be an equilibrium of
 Eq. (\ref{dynsys}); then, we will argue that an upper bound  on the convergence time of Eq. (\ref{dynsys}) would imply a similar convergence time bound on the linearization of Eq. (\ref{dynsys}) around the equilibrium of $\0$.
This will allow us to apply a previous $\Omega(n^2)$ convergence time lower
bound for {\em linear} schemes, proved by the authors in \cite{OT09}.

Let $f$ (without a subscript) be the mapping from $\R^n$ to itself that
maps $x(t)$ to $x(t+1)$ according to Eq.\ (\ref{dynsys}).

\begin{lemma} $f(a\1)=a\1$, for any $a \in \R$. \label{fixedpoint} \end{lemma}
\begin{proof} Suppose that $x(0)=a \1$. Then, the initial average is $a$, so that
 \[ a \1 = \lim_t x(t) = \lim_t x(t+1) = \lim_t f(x(t)).\] We use the continuity of $f$ to get
 \[ a \1 = f(\lim_t x(t)) = f(a \1). \]
\end{proof}

%

For $i,j=1,\ldots,n$, we define  $a_{ij} = \frac{\partial
f_i(0)}{\partial x_j}$,
and the matrix
\[ A=  f'({\bf 0}) = \left(
                         \begin{array}{cccccc}
                           a_{11} & a_{12} & 0 & 0 & \cdots & 0 \\
                           a_{21} & a_{22} & a_{23} & 0 & \cdots & 0 \\
                           0 & a_{32} & a_{33} & a_{34} & \cdots & 0 \\
                           \vdots & \vdots & \vdots & \vdots & \vdots & \vdots \\
                           0 & \cdots & 0 & 0 & a_{n,n-1} & a_{nn} \\
                         \end{array}
                       \right). \]

\begin{lemma} For any integer $k \geq 1$, \[ \lim_{x \rightarrow \0} \frac{\|f^{k}(x) - A^k x\|_2}{\|x\|_2} = 0,  \]
\label{compositionapprox} \aa{where $f^{k}$ refers to the $k$-fold composition of $f$ with itself.} \end{lemma} \begin{proof} The fact that
$f(\0)=\0$ implies by the chain rule that the derivative of $f^k$ at
$x=\0$ is $A^k$. The above equation is a restatement of this fact.
\end{proof}
\begin{lemma} Suppose that $x^T \1=0$. Then, \[ \lim_{m \rightarrow \infty} A^m x = {\bf 0}.\] \label{conv}
\end{lemma} \begin{proof} \ao{Let ${\cal B}$ be a ball around the origin such that for all $x \in {\cal B}$, with $x\neq {\bf 0}$, we have}
\[ \frac{\|f^k ( x )- A^k x \|_2}{\|x\|_2} \leq \frac{1}{4}, ~~\mbox{ for  } k=\tau(n,1/2). \] \ao{Such a ball can be found due to  Lemma \ref{compositionapprox}. Since we can scale $x$ without affecting the assumptions or conclusions of the lemma we are trying to prove, we can assume that $x \in {\cal B}$.}
It follows that that for $k=\tau(n,1/2)$, we have
\begin{eqnarray*} \frac{\|A^{k} x\|_2}{\|x\|_2} & = &   \frac{\|A^{k} x - f^{k}(x) + f^{k} (x)\|_2}{\|x\|_2} \\
 & \leq &  \frac{1}{4} + \frac{\|f^{k} (x)\|_2}{\|x\|_2} \\
  & \leq &  \frac{1}{4} + \frac{1}{2} \\
  &  \leq & \frac{3}{4}.
   \end{eqnarray*}  \ao{Since this inequality implies that $A^{k} x \in {\cal B}$, we can apply} the same argumet argument recursively to get
\[ \lim_{m \rightarrow \infty} (A^{k})^m x = \0,\]  which implies the conclusion of the lemma.
\end{proof}
\bigskip
\begin{lemma} A\1 = \1. \label{righteigen} \end{lemma}
\begin{proof} We have \[ A \1 = \lim_{h \rightarrow 0} \frac{f(\0 + h  \1)-f(\0)}{h} = \lim_{h \rightarrow 0} \frac{ h \1}{h} = \1,\] where we used Lemma \ref{fixedpoint}. \end{proof}

\begin{lemma} For every vector $x \in \R^n$, \[ \lim_{k \rightarrow \infty} A^k x = {\bar x} \1 ,\]
where ${\bar x}=(\sum_{i=1}^n x_i)/n$.
\label{convlemma}
\end{lemma}
\begin{proof} Every vector $x$ can be written as
\[ x = 
{\bar x} \1  + y,\] where $y^T \1 = 0$. Thus,
\[ \lim_{k \rightarrow \infty} A^k x = \lim_{k \rightarrow \infty} A^k \left( {\bar x} \1  + y \right) =
{\bar x} \1 + \lim_{k \rightarrow \infty} A^k y = {\bar x} \1, \] where we used
Lemmas \ref{conv} and \ref{righteigen}.
\end{proof}

\bigskip

\begin{lemma} \label{aproperties} The matrix $A$ has the following properties:
\begin{enumerate} \item $a_{ij}=0$ whenever $|i-j| >1$.
\item The graph $G = (\{1,\ldots,n\},E)$, with $E = \{ (i,j) ~|~ a_{ij} \neq 0 \}$, is strongly connected.
\item $A \1 = \1$ and $\1^T A = \1$.
\item An eigenvalue of $A$ of largest modulus has modulus $1$.
\item $A$ has an eigenvector $v$, with real eigenvalue $\lambda \in (1-\frac{6}{n^2},1)$, such that $v^T {\bf 1}=0$.
\end{enumerate}
\end{lemma}

\begin{proof} \begin{enumerate} \item True because of the definitions of $f$ and $A$.
\item Suppose not. Then, there is a nonempty set $S \subset \{1,\ldots,n\}$ with the property that
$a_{ij}=0$ whenever $i\in S$ and $j\in S^c$. Consider the vector $x$ with
$x_i=0$ for $i \in S$, and $x_j=1$ for $j \in S^c$. Clearly, $(1/n) \sum_i x_i > 0$, but
$(A^k x)_i=0$ for $i \in S$. This contradicts Lemma \ref{convlemma}.
\item The first equality was already proven in Lemma \ref{righteigen}. For the second, let $b=\1^T A$. Consider the
vector \begin{equation} z =  \lim_{k \rightarrow \infty} A^k e_i, \end{equation} where $e_i$ is the $i$th unit
 vector. By Lemma \ref{convlemma}, \[
z = \frac{\1^T e_i}{n} \1 = \frac{1}{n} \1.\] On the
other hand, \[ \lim_{k \rightarrow \infty} A^k e_i = \lim_{k \rightarrow \infty} A^{k+1} e_i = \lim_{k \rightarrow \infty} A^k (A e_i).\] Applying Lemma \ref{convlemma} again, we get \[
z = \frac{\1^T \cdot (A e_i)}{n} \1 = \frac{b_i}{n} \1, \] where $b_i$ is the $i$th component of $b$. We conclude
that $b_i=1$; since no assumption was made on $i$, this implies that $b=\1$, which is what we needed to show.
\item We already know that $ A \1 = \1$, so that an eigenvalue with modulus $1$ exists.
Now suppose there is an eigenvalue with larger modulus, that is, there is some vector $x \in \mathbb{C}^n$ such that
 $Ax = \lambda x$ and $|\lambda| > 1$. Then $\lim_k \|A^k x\|_2 = \infty$. By writing $x = x_{\rm real} + i x_{\rm imaginary}$, we immediately have that $A^k x = A^k x_{\rm real} + i A^k x_{\rm imaginary}$. But by Lemma \ref{convlemma} both $A^k x_{\rm real}$ and $A^k x_{\rm imaginary}$ approach some finite multiple of $\1$ as $k \rightarrow \infty$, so $\|A^k x\|_2$ is bounded above. This is a contradiction.
\item \ao{The following fact  \aa{is a combination of Theorems 4.1 and 6.1 in \cite{OT09}}: Consider an $n \times n$ matrix $A$ such that $a_{ij}=0$ whenever $|i-j|>1$, and such that the graph with edge set $\{ (i,j) ~|~ a_{ij} \neq 0 \}$ is connected. Let $\lambda_1, \lambda_2, \ldots$ be its eigenvalues in order of decreasing modulus. Suppose that $\lambda_1=1$, $A \1 = \1$, and $\pi^T A = \pi^T$, for some vector $\pi$ satisfying $\sum_i \pi_i = 1$, and $\pi_i \geq 1/(Cn)$ for some positive $C$ and for all $i$. Then, $A$ has a \aa{real} eigenvalue in\footnote{\ao{The reference
\cite{OT09} proves that an eigenvalue lies in $(1-c_1 C/n^2,1)$ for some absolute constant $c_1$. By tracing through the proof, we find that we can take $c_1=6$.}} $(1-6C/n^2,1)$.}
Furthermore,  the corresponding eigenvector is orthogonal to ${\bf 1}$, since right-eigenvectors of a matrix are
orthogonal to left-eigenvectors with different eigenvalues.

\ao{By parts 1-4, all the assumptions of the result from \cite{OT09} are satisfied with $\pi=\1/n$ and $C=1$, thus completing the proof
 of the lemma.}
\end{enumerate}
\end{proof}

\bigskip

\noindent \ao{ {\bf Remark:}  An alternative proof of part 5 is possible.
One can argue that parts 1 and 3 force $A$ to be symmetric, and that Lemma \ref{convlemma} implies that
the elements $a_{ij}$ must be nonnegative. Once these two facts are established, the results of \cite{BDSX06} will then imply
 an eigenvalue has to lie in $(1-c/n^2,1)$,
 for a certain absolute constant $c$.
  }

\bigskip

\noindent {\bf Proof of Theorem \ref{mainthm}}:  Let $v$ be an eigenvector of $A$ with the properties in  part 5 of Lemma \ref{aproperties}. Fix a positive integer $k$. Let $\epsilon>0$ and pick $x \neq \0$ to be a small enough multiple of $v$ so that \[ \frac{\|f^k(x)-A^k(x)\|_2}{\|x\|_2} \leq \epsilon. \] This is possible by Lemma \ref{compositionapprox}.
Then, we have \[ \frac{\|f^k(x)\|_2}{\|x\|_2} \geq \frac{\|A^k x\|_2}{\|x\|_2} - \epsilon \geq  \left( 1-\frac{6}{n^2} \right) ^k - \epsilon. \] Using the orthogonality property $x^T {\bf 1} =0$, we have ${\bar x}=0$. Since we placed no restriction on $\epsilon$, this implies that
\[ \inf_{x \neq 0} \frac{\|f^k(x)-{\bar x}{\bf 1}\|_2}{\|x -{\bar x}{\bf 1}\|_2}=
\inf_{x \neq 0} \frac{\|f^k(x)\|_2}{\|x\|_2} \geq \left( 1 - \frac{6}{n^2}\right)^k \] Plugging $k=\tau(n,\epsilon)$ into this equation, we see that
\[ \left(1- \frac{6}{n^2}\right)^{\tau(n,\epsilon)} \leq \epsilon.\] Since $n \geq 3$, we have $1-6/n^2 \in (0,1)$, and
\[ \tau(n,\epsilon) \geq \frac{1}{\log (1-6/n^2)} \log \epsilon.\] Now using the bound $\log(1-\alpha) \geq 5(\alpha-1)$ for $\alpha \in [0,2/3)$, we get
\[ \tau(n,\epsilon) \geq \frac{n^2}{30} \log \frac{1}{\epsilon}.\] \noindent \qed
\bigskip

\noindent {\bf Remark:} \ao{We now sketch the proof of
 the claim we made earlier that a local averaging algorithm cannot average in finite time. \aa{Fix $n \geq 3$.} Suppose that for any $x(0)$ in some ball ${\cal B}$ around the origin, a local averaging algorithm results in $x(t) = {\bar x} \bf 1$ for all $t \geq T$.}

\ao{The proof of Theorem 1 shows that given any $k, \epsilon >0$, one can pick a vector $v(\epsilon)$ so that if $x(0)=v(\epsilon)$ then \aa{ $V(x(k))/V(x(0)) \geq (1-6/n^2)^k - \epsilon$. Moreover, the vectors $v(\epsilon)$ can be chosen
to be arbitrarily small. One simply picks $k = T$ and $\epsilon < (1-6/n^2)^k$ to get that $x(T)$ is not a multiple of $\1$; and furthermore, picking $v(\epsilon)$ small enough in norm to be in ${\cal B}$ results in a contradiction. }

\bigskip

\noindent {\bf Remark:} Theorem \ref{mainthm} gives a lower bound on how long we must wait for the $2$-norm $\|x(t)-\bar{x} \1\|_2$ to
shrink by a factor of $\epsilon$. What if we replace the $2$-norm with other norms, for example with the $\infty$-norm? Since
 ${\cal B}_{\infty}(\0, r/\sqrt{n}) \subset {\cal B}_{2}(\0,r) \subset {\cal B}_{\infty}(\0, r)$, it follows that if the $\infty$-norm shrinks by a factor of $\epsilon$, then the $2$-norm must shrink by at least $\sqrt{n} \epsilon$. Since $\epsilon$ only enters the lower bound of Theorem \ref{mainthm} logarithmically, the  answer only changes by a factor of $\log n$ in passing to the $\infty$-norm. A similar argument shows that, modulo some logarithmic factors, it makes no difference which $p$-norm is used.

\section{Conclusions}

We have proved a lower bound on the convergence time of local averaging algorithms which scales
 quadratically in the number of agents. This lower bound holds even if all the communication graphs are equal to a
 fixed line graph.  Our work points to a number of open questions.

\begin{enumerate} \item Is it possible to loosen the definition of local averaging algorithms to encompass a wider class of
algorithms? In particular, is it possible to weaken the requirement that each $f_{i,G(t)}$ be smooth, perhaps only
to the requirement that it be piecewise-smooth or continuous, and
still obtain a $\Omega(n^2)$ lower bound?
\item Does the worst-case convergence time change if we introduce some memory and allow
$x_i(t+1)$ to depend on the last $k$ sets of messages received by
agent $i$? Alternatively, there is the broader question of how much is
there to be gained if every agent is allowed to keep track of extra
variables. Some positive results in this direction were obtained
in \cite{JSS09}.
\item What if each node maintains a small
number of update functions, and is allowed to choose which of them to apply? Our lower
bound does not apply to such schemes, so it is an open question whether its possible
to design practical algorithms along these lines with worst-case convergence 
time scaling better than $n^2$. 
\end{enumerate}


\begin{thebibliography}{1}

\bibitem{AH08} M. Alighanbari, J.P. How, ``Unbiased Kalman Consensus Algorithm,'' {\em Proceedings of the 2006 American Control Conference, } Minneapolis, Minnesota, USA, June 14-16, 2006

\bibitem{BCH08} L. Brunet, H. L. Choi, J. P. How, ``Consensus-based auction approaches for decentralized task
assignment,'' {\em  AIAA Guidance, Navigation, and Control
Conference,} Honolulu, Hawaii, Aug. 2008.

\bibitem{BHOT05} V.\ D.\ Blondel, J.\ M.\ Hendrickx, A.\ Olshevsky, and J.\ N.
Tsitsiklis, ``Convergence in multiagent coordination, consensus, and
flocking,'' in {\it Proceedings of the Joint 44th IEEE Conference on
Decision and Control and European Control Conference (CDC-ECC'05)},
Seville, Spain, December 2005.

\bibitem{BDSX06} S. Boyd, P. Diaconis, J. Sun, and L. Xiao, ``Fastest mixing Markov chain on a path,'' {\em
The American Mathematical Monthly}, 113(1):70-74, January 2006.

\bibitem{CMA08} M. Cao, A.S. Morse, B.D.O Anderson, ``Reaching a Consensus in a Dynamically Changing Environment: Convergence Rates, Measurement Delays, and Asynchronous Events,'' {\em SIAM Journal on Control and Optimization,}
    47(2):601-623, 2008.


\bibitem{CCSZ08} R. Carli, A. Chiuso, L. Schenato, S. Zampieri, ``Distributed Kalman filtering based on consensus strategies,'' {\em IEEE Journal on Selected Areas in Communications,} 26(4):622-633, 2008.

\bibitem{C06} J. Cortes, ``Finite-time convergent gradient flows with applications to network consensus, ''
{\em Automatica,} Vol. 42(11):1993-2000, 2006.

\bibitem{GCB08} C. Gao,  J. Cortes, F. Bullo, ``Notes on averaging over acyclic digraphs and discrete coverage
control,'' {\em  Automatica,} 44(8):2120-2127, 2008.

\bibitem{HU08} N. Hayashi, T. Ushio, ``Application of a consensus problem to fair multi-resource allocation
in real-time systems, '' {\em Proceedings of the 47th IEEE Conference on Decision and Control,} Cancun, Mexico, 2008.

\bibitem{JLM03} A. Jadbabaie, J. Lin, A.S. Morse, ``Coordination of groups of mobile autonomous agents using nearest neighbor rules,'' {\it IEEE Transactions on Automatic Control,} 48(6):988-1001, 2003.


\bibitem{JSS09} K. Jung, D. Shah, J. Shin, ``Distributed averaging via lifted Markov chains, '' preprint, 2008.

\bibitem{KBS07} A. Kashyap, T. Ba\c{s}ar, R. Srikant, ``Quantized consensus,'' {\em Automatica,} 43(7):1192-1203,
2007.

\bibitem{LR04} Q. Li, D. Rus, ``Global clock synchronization for sensor networks, ''
{\em Proceedings of Infocom,} Hong Kong, March 2004.

\bibitem{LW04} S.\ Li and H.\ Wang, ``Multi-agent coordination using
nearest-neighbor rules: revisiting the Vicsek model,'' 2004;
\url{http://arxiv.org/abs/cs.MA/0407021}.

\bibitem{M05} L. Moreau, ``Stability of multiagent systems with time-dependent communication links,''
    {\em IEEE Transactions on Automatic Control, } 50(2):169-182, 2005.

\bibitem{MVR06} C. C. Moallemi and B. Van Roy, ``Consensus propagation,'' {\it IEEE Transactions on Information Theory,} 52(11):4753-4766, 2006.

\bibitem{NOOT07} A. Nedic, A. Olshevsky, A. Ozdaglar, and J. N. Tsitsiklis. ``On distributed
averaging algorithms and quantization effects,'' {\em IEEE Transactions on Automatic Control,} 54(11):2506-2517, 2009.

\bibitem{OM04} R. Olfati-Saber and R. M. Murray.``Consensus problems in networks of agents with switching topology and time-delays,'' {\em IEEE Trans. on Automatic Control}, 49(9):1520-1533, Sep., 2004.

\bibitem{SH07} S. Sundaram, C.N. Hadjicostis, ``Finite-time distributed consensus in graphs with time-invariant
topologies,'' {\em Proceedings of the American Control Conference,} New York, NY, July 2007.

\bibitem{SSR08} M. Schwager, J.-J. Slotine, D. Rus, ``Consensus learning for distributed coverage control,''
{\em  Proceedings of International Conference on Robotics an
Automation,} Pasadena, CA, May 2008.


\bibitem{OT09} A. Olshevsky, J.N. Tsitsiklis, ''Convergence speed in distributed consensus and averaging,'' {\em SIAM Journal on Control and Optimization}, Volume 48(1):33-55, 2009.

\bibitem{TBA86} J. N. Tsitsiklis, D. P. Bertsekas and M. Athans, ``Distributed asynchronous deterministic and stochastic gradient optimization algorithms,''  {\it IEEE Transactions on Automatic Control}, 31(9):803-812, 1986.

\bibitem{WSD09} F. Wuhid, R. Stadler, M. Dam, ``Gossiping for threshold detection, '' {\em
Proceedings of the 11th IFIP/IEEE international conference on Symposium on Integrated Network Management,} 2009.

\bibitem{XB04} L. Xiao and S. Boyd, ``Fast linear iterations for distributed averaging, '' {\em Systems and Control Letters,} 53:65-78, 2004.

\bibitem{XBL05} L. Xiao, S. Boyd and S. Lall, ``A Scheme for robust distributed sensor fusion based on average consensus, '' {\em Proceedings of International Conference on Information Processing in Sensor Networks,} April 2005, p63-70, Los Angeles, 2005.

\bibitem{ZM08} M. Zhu, S. Martinez, ``On the convergence time of asynchronous distributed quantized averaging algorithms,'' preprint, 2008.

\end{thebibliography}
\end{document}